\documentclass[12pt]{article}      

\usepackage[title]{appendix}

\usepackage{url}

\usepackage{amsmath}
\usepackage{amsthm}
\usepackage{amssymb}

\newcommand{\R}{\ensuremath{\mathbb{R}}}

\renewcommand{\P}{\ensuremath{\mathbb{P}}}

\newtheorem{theorem}{Theorem}[section]
\newtheorem{lemma}[theorem]{Lemma}
\newtheorem{corollary}[theorem]{Corollary}
\newtheorem{proposition}[theorem]{Proposition}
\newtheorem{remark}[theorem]{Remark}

\newtheorem{definition}[theorem]{Definition}

\usepackage{color}
\definecolor{maroon}{rgb}{.69,.188,.376}
\definecolor{darkgreen}{rgb}{0,.5,0}
\definecolor{darkblue}{rgb}{0,0,.5}
\definecolor{magenta}{rgb}{1,0,1}

\usepackage[colorlinks=true]{hyperref}
\definecolor{Red}{rgb}{1,0,0}
\definecolor{Blue}{rgb}{0,0,1}
\definecolor{Olive}{rgb}{0.41,0.55,0.13}
\definecolor{Yarok}{rgb}{0,0.5,0}
\definecolor{Green}{rgb}{0,1,0}
\definecolor{MGreen}{rgb}{0,0.8,0}
\definecolor{DGreen}{rgb}{0,0.55,0}
\definecolor{Yellow}{rgb}{1,1,0}
\definecolor{Cyan}{rgb}{0,1,1}
\definecolor{Magenta}{rgb}{1,0,1}
\definecolor{Orange}{rgb}{1,.5,0}
\definecolor{Violet}{rgb}{.5,0,.5}
\definecolor{Purple}{rgb}{.75,0,.25}
\definecolor{Brown}{rgb}{.75,.5,.25}
\definecolor{Grey}{rgb}{.7,.7,.7}
\definecolor{Black}{rgb}{0,0,0}

\hypersetup{pdftex, colorlinks=true, linkcolor=maroon, citecolor=maroon,
  filecolor=blue,urlcolor=blue}

\newcommand{\ignore}[1]{{}}

\numberwithin{equation}{section}

\date{\today}

\begin{document}

\title{Smoothness of Flow and Path-by-Path Uniqueness in Stochastic Differential Equations}

\author{Siva Athreya
  \and
  Suprio Bhar
  \and  Atul Shekhar
}

\maketitle
\begin{abstract}
We consider the stochastic differential equation
$$ X_t = x_0 + \int_0^t f(X_s)ds + \int_0^t\sigma(X_s)dB^{H}_s,$$ with $x_0 \in \R^d$, $d \geq 1$, $f: \R^d \rightarrow
\R^d$ is bounded continuous, $\sigma: \R^d \rightarrow \R^{d\times d}$ is a
uniformly elliptic, bounded, twice continuously differentiable
conservative vector field and $B^H$ is fractional Brownian motion with
$H \in (\frac{1}{3}, \frac{1}{2}]$. When $d=1$, $H= \frac{1}{2}$, and
  $f$ is H\"older continuous, in the spirit of Davie \cite{davie}, we establish the existence of a null set
  $\mathcal{N}$ depending only on $f, \sigma$ such that for all $x_0\in
  \mathbb{R}$ and $\omega \in \Omega\setminus \mathcal{N}$, the above equation admits a path-by-path unique
  solution. Our proof is based on establishing the uniform continuous differentiability of the flow associated with the
  equation. We also establish the path-by-path uniqueness for $d \geq 1$ and $H \in (\frac{1}{3}, \frac{1}{2}]$, but the null set may depend on $x_0$, thus extending a
    result of Catellier-Gubinelli \cite{CG}.
\end{abstract}

\noindent {\em AMS 2010 Subject Classification :}  60J65,  60H10, 60G17.\\
\noindent {\em Keywords :} Fractional Brownian motion, Rough Differential Equations, Rough Paths,  Path by Path Uniqueness, Smooth Flow.

\section{Introduction}

In this article we consider stochastic differential equations (SDEs) of the
form,
\begin{equation}
\label{modelSDE} X_t =  x_0 + \int_0^tf(X_s)ds + \int_0^t\sigma(X_s)d{B}^H_s
\end{equation}
with $x_0 \in \R^d$, $d \geq 1$, $f: \R^d \rightarrow \R^d$ is bounded continuous,
$\sigma: \R^d \rightarrow \R^{d\times d}$ is a uniformly elliptic, bounded,
twice continously differentiable conservative vector field and $B^H$
is fractional Brownian motion with $H \in (\frac{1}{3}, \frac{1}{2}]$.

When $\sigma \equiv 0$, \eqref{modelSDE} reduces to a differential
equation. It is well known that such differential equations need not
have a unique solution. However  for non-zero $\sigma$ the phenomenon called
regularisation by noise or attributed as the ``regularizing
 effect of quadratic variation (resp. local time) of Brownian motion (resp. fractional Brownian motion)"  as mentioned in
 [\cite{RY}-Chapter 9-section 3] comes into play.

This was demonstrated for $\sigma \equiv 1$, $H=\frac{1}{2}$ (then $B^H$ is
the standard Brownian motion) in the seminal paper \cite{davie} by A.M. Davie. The case of $\sigma \equiv 1$ and general $H \in (0,1) $ was also understood deeply by Catellier and Gubinelli in \cite{CG}. In such cases the above equation reduces to
\begin{equation}
\label{modelSDEs1} X_t =  x_0 + \int_0^tf(X_s)ds + {B}^H_t.
\end{equation}
The existence and uniqueness of the solution for the (\ref{modelSDEs1}) interpreted as an ordinary differential equation (ODE)
was shown in the respective cases in
\cite{davie, CG} ($f$ is taken to be only bounded measurable in \cite{davie}). In both of the above papers it is shown that there exists
a null set ${\cal N}$ under the law of $B^H$ depending on $x_0$ and $f$ such that for each
$\omega \notin {\cal N}$, (\ref{modelSDEs1}) has a
unique solution in the space of continuous curves. Such uniqueness of solution is referred to as
\textit{path-by-path} uniqueness as coined by Flandoli in
\cite{F}. This is a stronger notion of uniqueness than the strong uniqueness of solutions to SDEs where the SDEs are solved in the space of adapted processes. \

The regularizing effect of Brownian motion has been explored earlier
using It\^o Calculus by Veretennikov \cite{V81} for more general $\sigma$, Krylov and Roeckner
\cite{KR} on singular stochastic differential equations and work of
Flandoli, Gubinell and Priola \cite{FGP} on transport equations. This regularisation effect  has been extended to infinite dimensional setting of stochastic heat equation by Butkovsky and Mytnik in \cite{BM}. More recently, using the work of Fedrizzi and Flandoli in \cite{Flandoli-Fedrizzi}, Shaposhnikov in \cite{S1,S2} presented a simpler proof of Davie's result by establishing a H\"older regularity on the flow associated to equation \eqref{modelSDEs1}.

In \cite{davie1}, (\ref{modelSDE}) is considered with $H=\frac{1}{2}$,
bounded measurable $f$ and $\sigma$ to be an invertible matrix-valued
function satisfying some regularity conditions. A sketch of proof for
path-by-path uniqueness is presented, interpreting the
(\ref{modelSDE}) in a ``rough path" sense. However, to the best of our
knowledge, it is not possible to convert the outlined sketch to a 
precise proof.

In this paper we consider (\ref{modelSDE}) as a rough differential
equation (RDE) (see (\ref{modelRDE})). For $d=1$, $H= \frac{1}{2}$,
and when $f$ is H\"older continuous, we establish the existence of a
null set depending only on $f, \sigma$ such that for all $x_0 \in
\mathbb{R}$ and $\omega \in \Omega\setminus \mathcal{N}$, the
(\ref{modelSDE}) viewed as a RDE admits a unique solution (See Theorem
\ref{pbypins}). For proving this, using techniques from rough path
theory, we develop a variant of It\^o formula and use it to construct
a bijective diffeomorphism that transforms the equation
\eqref{modelSDE} to the case of $\sigma \equiv 1$ (See Proposition
\ref{transformation}). The proof then follows from establishing the
continuous differentiability of the flow associated with the equation
(see Theorem \ref{just-flow}). Continuous differentiability of the flow
is proved by building upon a result on existence of flow by Fedrizzi
and Flandoli in \cite{Flandoli-Fedrizzi} for compactly supported
functions $f$. We establish an exponential identity for this flow, see
equation \eqref{derivative-identity}.  We then show that the terms in
the exponential are stochastic integrals with continuous modifications
(See Proposition \ref{exp-identity}), yielding smooth modification for
the flow constructed from \cite{Flandoli-Fedrizzi}. The case of the
general $f$ follows via a localisation argument.

Moreover, the rough It\^o formula mentioned above works well in general dimensions $d\geq 1$ and for $H
\in (\frac{1}{3}, \frac{1}{2}],$ and we use it show that the (\ref{modelSDE}) viewed as a RDE has a path by path unique solution outside a null set of
  $B^H$, depending on $x_0, f, \sigma$ (see Theorem \ref{gen-catellier-gubinelli}). The argument for the smoothness of the flow  which works well for $d=1$ however doesn't easily generalize to the case of higher dimensions. The techniques required to emulate this approach are more involved in higher dimensions. We plan to carry this out in a future project.

\subsection{Model and Main results}
 Let $d \geq 1, x_0 \in \R^d$ and $M_{d\times d}(\R)$ be $d$-dimensional matrices with real entries.  Let $f: \R^d \rightarrow \R^d$ and $\sigma: \R^d \rightarrow M_{d \times d}(\R)$ such that
\begin{itemize}

\item[($F1$)] $f:\R^d \rightarrow \R^d$ is a bounded continuous function.

\item[($S1$)] $\sigma: \R^d \rightarrow M_{d\times d}(\R) $ is continuously twice differentiable and  \[||\sigma||_{\infty} + ||D\sigma||_{\infty} + ||D^2\sigma||_{\infty} < \infty,\]
where $\parallel \cdot \parallel_\infty$ is the supremum norm, $D\sigma$ is the total derivative  of $\sigma$ and $D^2\sigma$ is the Hessian of $\sigma$.

\item[($S2$)] $\sigma$ is uniformly elliptic, i.e. there exist a constant $\lambda > 0$ such that for all $v, x \in \mathbb{R}^d$, 
\[ |v^{\intercal}\sigma(x)v| \geq \lambda ||v||^2.\]
 In particular, $\sigma(x)$ is an invertible matrix for all $x\in \R^d$.

\item[($S3$)] $\sigma^{-1}$ is a conservative vector field, i.e. the path integral  
\[ \int_{\gamma} \sigma(x)^{-1}dx \]
depends only on the end points of path $\gamma$ for any given path $\gamma$ in $\mathbb{R}^d$.

\end{itemize}

\begin{definition}[Fractional Brownian motion] \label{d-fbm}
Let $(\Omega, {\cal F} , \{{\cal F}_t\}_{t \geq 0} , \P)$ be a complete filtered probability space. The fractional Brownian motion $\{B_t^H\}_{t\geq 0}$ with the Hurst parameter $H$ is the unique adapted and centered Gaussian process with values in $\mathbb{R}^d$ and covariance matrix $C^H(s,t)$ given by 
\[ C^H(s,t) = \frac{1}{2}(t^{2H} + s^{2H} - |t-s|^{2H})I_d,\]
where $I_d$ is the $d\times d$ identity matrix. 
\end{definition}

The processes $B^H$ can naturally be lifted to a rough path almost surely and one can use such rough paths to define integration against $B^H$, see Section \ref{a-prps} for a brief introduction to rough path theory. We will use rough path theory to give a deterministic interpretation of equation \eqref{modelSDE} by equation \eqref{modelRDE} below. We also note that use of rough path theory to give a deterministic interpretation of stochastic integrals can also be considered as a generalization of work of F\"ollmer \cite{follmer} and Karandikar \cite{karan}.

\begin{lemma} \label{rough-path-fbm} Let $H \in (\frac{1}{3}, \frac{1}{2}]$ and $\alpha \in  (\frac{1}{3}, H)$. Let  $B^H$ be the fractional Brownian motion defined above. There exists $\Omega^\prime \subset \Omega$ with $\P(\Omega^\prime) =1 $ and objects $\mathbf{B}^H(\omega) = (B^H(\omega), \mathbb{B}^H(\omega))$  such that for each $\omega \in \Omega^{\prime}$, $\mathbf{B}^H(\omega)$ is an $\alpha$ geometric rough path.   
\end{lemma}
\begin{proof}  The rough path $\mathbf{B}^H$ was constructed in
  \cite{Coutin-Qian}. See \cite{Coutin-Qian}, \cite{FV07} or Chapter $10$ of
  \cite{FH12} for details. 
  \end{proof}

The choice of the object $\mathbf{B}^H$ enhancing $B^H$ to a geometric rough path is of course non-unique, e.g. for any anti symmetric matrix $M$, $\mathbb{B}^{H,M}_{s,t} = \mathbb{B}^{H}_{s,t} + M (t-s) $ defines a yet another geometric rough path. We will fix a choice of a geometric rough path $\mathbf{B}^{H}$ in below. For the special case $H= \frac{1}{2}$, we make the choice $\mathbf{B}^H = \mathbf{B}^{Strat}$ defined by 
\[\mathbb{B}_{s,t}^{Strat} := \int_s^t \{B_r -B_s \}\otimes \circ dB_r,\]
where $\circ dB_r$ denotes the Stratonovich integral. We will also consider a non-geometric rough path $\mathbf{B}^{Ito}$ defined by \[\mathbb{B}_{s,t}^{Ito} := \int_s^t \{B_r -B_s \}\otimes dB_r,\]
where $dB_r$ denotes the usual It\^o integral. As we will see, the results of this paper is independent of the choices of the rough paths being made here.  

Let $C^{\alpha}_{B^H}(\mathbb{R}^d)$ be the space of $B^H$-controlled rough paths taking values in $\mathbb{R}^d$. Define the map $\Psi: C^{\alpha}_{B^H}(\mathbb{R}^d) \to
C^{\alpha}_{B^H}(\mathbb{R}^d)$ by
\begin{equation} \label{Psi} \Psi(Y, Y') (t):= \biggl(x_0 + \int_0^t f(Y_r)dr + \int_0^t \sigma(Y _r)d\mathbf{B}^H_r , \sigma(Y_t) \biggr)
\end{equation}
for $t \in [0,T]$. Note that first integral in the definition of
$\Psi$ is a Riemann integral and the second integral is the rough
integral of controlled rough path $(\sigma(Y), \sigma'(Y)Y')$ against
the rough path $\mathbf{B}^H$. Under assumptions (F1) and (S1) it is
easy to check that $\Psi(Y, Y')$ is indeed a controlled rough path and
$\Psi$ is well defined.

A controlled rough path $(X,X')\in C^{\alpha}_{B^H}(\mathbb{R}^d)$ is called a solution to 
\begin{equation}
\label{modelRDE} X_t =  x_0 + \int_0^tf(X_s)ds + \int_0^t\sigma(X_s)d\mathbf{B}^H_s
\end{equation}
if it is a fixed point of $\Psi$, i.e $\Psi(X,X^\prime) = (X,
X^\prime)$. The uniqueness of solution to \eqref{modelRDE} means that for any two solutions $(X, X')$ and $(\tilde{X}, \tilde{X}')$ to \eqref{modelRDE}, the equality $ (X, X') =(\tilde{X}, \tilde{X}')$ holds. Further, since fractional Brownian motions are truly rough, see \cite{FS1}, by \eqref{doob-meyer-thm}, establishing $X= \tilde{X}$ for any two
solutions $(X,X')$ and $(\tilde{X}, \tilde{X}')$ of equation
\eqref{modelRDE} gives the uniqueness of solution to
\eqref{modelRDE}. Also, if $(X,X')$ is a solution to \eqref{modelRDE},
then it implies that $(X, \sigma(X))$ is also a controlled rough path
and thus $X' = \sigma(X)$. We will exploit this fact very crucially.

We are now ready to state our main results. Our first main result is
about path by path existence and uniqueness for (\ref{modelRDE}). The
term \textit{path-by-path} existence and uniqueness of the solution
was coined by F. Flandoli in \cite{F} and refers to the following
definition. Recall that a set $\mathcal{N} \subset \Omega$ is called a
null set if $\mathbb{P}(\mathcal{N}) = 0$.

\begin{definition} \label{pbypeu}
The equation \eqref{modelRDE} is said to have a path-by-path unique
solution if there exist a null set $\mathcal{N}$ depending on
$f,\sigma$ and $x_0$ such that for all $\omega \notin \mathcal{N}$, the 
equation \eqref{modelRDE} has a unique solution.
\end{definition}

\begin{theorem} \label{gen-catellier-gubinelli}  Let $d \geq 1$, $x_0 \in \R^d$. Let $B^H$ be the fractional Brownian motion with $H \in (\frac{1}{3}, \frac{1}{2}]$.  Assume $(F1), (S1)-(S3)$. The equation \eqref{modelRDE} has a path-by-path unique solution. 
\end{theorem}

Now we restrict ourselves to the case of $H= \frac{1}{2}$ and
$d=1$. In this case, we consider the choice of rough path $\mathbf{B}= \mathbf{B}^{Ito}$ or $\mathbf{B} = \mathbf{B}^{Strat}$. Our second main result is about the smoothness of the flow
associated to equation \eqref{modelRDE}. We first define the meaning of the flow associated to \eqref{modelRDE} as follows. With a slight abuse of notation, we write $\psi(s,t,x)$ to mean both as a real valued function in $s,t,x$ or a controlled rough path  $(\psi(s,t,x),\sigma(\psi(s,t,x))$ in variable $t$.

\begin{definition} \label{just-flow} Let $x\in\R$ and $0 \leq s
  \leq t < T$.   The stochastic flow $\psi(s,t,x)\equiv \psi(s,t,x,\omega)$ of the equation \eqref{modelRDE} defined on $\{0\leq s \leq t < T\}\times\R\times\Omega$ is a jointly measurable collection of random variables on $(\Omega, {\cal F}, \P)$ satisfying:
\begin{enumerate}
\item Almost surely, the map $(s,t,x) \mapsto \psi(s,t,x)$ is jointly
  continuous in $s,t,x$ .
\item Almost surely for all $x$ and $s\leq u \leq t$, $\psi(s,t,x) = \psi(u,t, \psi(s,u,x))$ and $\psi(s,s,x) = x$.
\item Almost surely for all $s, t, x$, $\psi(s,t,x)$  satisfies 
\[ \psi(s,t,x) = x + \int_s^t f(\psi(s,r,x))dr + \int_s^t \sigma(\psi(s,r,x))d\mathbf{B}_r.\]
\item For each $s,x$, the process $t\mapsto \psi(s,t,x)$ is adapted to the filtration of process  $t \mapsto B_t - B_s$.
\end{enumerate}

\end{definition}
Our first result is on establishing smoothness of the one dimensional flow.
\begin{theorem}\label{smooth-flow}
Let $d=1$ and $H= \frac{1}{2}$. Assume (F1), (S1)-(S2). Further assume that $f: \mathbb{R} \to \mathbb{R}$ is $\theta$-H\"older for some $\theta \in (0,1)$, i.e. 
 \[||f||_{\theta} := \sup_{x \neq y} \frac{|f(x) - f(y)|}{|x-y|^{\theta}} < \infty.\]
Then,
\begin{enumerate}
  \item  There exists a flow $\psi(s,t,x)$ satisfying
    Definition \ref{just-flow}.
    
 \item The flow $\psi(s,t,x)$ from Definition
   \ref{just-flow} is almost surely differentiable in $x$ for all $s,t,x$ and the map $
   D\psi(s,t,x) = \frac{d}{dx}\psi(s,t,x)$ is locally $\eta$-H\"older
   continuous for all $\eta \in [0, \frac{\theta}{2})$, i.e. for all
     compact sets $K \subset \{0\leq s \leq t < T\}\times \R$,
   \[\sup_{p,q \in K, p \neq q} \frac{|D\psi(p) - D\psi(q)|}{|p-q|^{\eta}} < \infty.\]

\end{enumerate}
\end{theorem}

As compared to various definitions of stochastic flows presented in
standard literature, Definition \ref{just-flow} is a slightly different
notion suitable for our purposes. For example, definition of a flow as
defined in \cite{Flandoli-Fedrizzi} (see Definition $5.1$ in there)
requires the flow to be a homeomorphism of the state space. We have
relaxed on this requirement in Definiton \ref{just-flow} and we will
distinguish between the \textit{flow} and \textit{flow of
  homeomorphism}. However, after establishing the above result we can
show that the homeomorphism property follows automatically from the
regularity properties of the flow. We summarise it as a corollary.

\begin{corollary} \label{likeflandoli} The field $\psi$ constructed above is also a flow of diffeomorphism. More precisely, almost surely for all $s, t$, the map $x \mapsto \psi(s,t, x)$ is a bijective continuously differentiable map with a continuously differentiable inverse. 
\end{corollary}

Our main result is a consequence of properties of the flow established above. This enables us to choose a null set independent of the starting point such that \eqref{modelRDE} has a path-by-path unique
solution.  

\begin{theorem} \label{pbypins} Let $d=1, H=\frac{1}{2}, x_0\in \R$. Assume (F1), (S1)-(S2). Further assume that $f$ is also $\theta$-H\"older for some $0 < \theta <1.$ Then the equation \eqref{modelRDE} has a path-by-path unique
solution with the null set $\mathcal{N}$ (as in Definition
\ref{pbypeu}) independent of $x_0\in \R$.
\end{theorem}

\begin{remark} \label{s-rem} We conclude this section with some remarks on the results presented above.
  \begin{itemize}
\item {\bf Rough Differential Equations:} Theorem \ref{gen-catellier-gubinelli} should also be compared to Theorem 9.1 of \cite{FH12} where it is shown
that the RDEs of form (\ref{modelSDE}) has a unique solution when $\sigma$ is continuously thrice differentiable and
$f$ is Lipschitz continuous. The apparent improvement in Theorem \ref{gen-catellier-gubinelli} is due to the fact that Theorem $9.1$ of \cite{FH12} is a result about general rough paths and regularizing effects of the underlying process is not accounted for. Theorem \ref{gen-catellier-gubinelli} however is very specific to fractional Brownian motions and would not hold for general RDEs.

  \item{\bf Strong Solution:}
    As indicated earlier the path by path uniqueness can be viewed as  regularizing effect of Brownian motion. This has been explored earlier
using It\^o Calculus by Veretennikov \cite{V81} for more general
$\sigma$, where a pathwise unique solution was established. We note that the solution that arises from Theorem \ref{pbypins} above will also be a strong solution and will be adapted to the filtration that governs the Brownian motion. This is the only point where we use the completeness of the Probability space stated in Definition \ref{d-fbm}.
\item {\bf Classical Time Change:} Our method of proof was motivated by the well known time-change idea. We present this  by considering a variant of the equation \eqref{modelSDE} in $d=1, H=\frac{1}{2}$ given by \eqref{timechange} below. While considering the weak uniqueness of solution (uniqueness in law) to \eqref{modelSDE}, one can use a martingale embedding theorem due to Dambis-Dubins-Schwarz which states that there exists a Brownian motion $W$ (called the DDS Brownian motion which depends on the solution $X$) such that 
\[ \int_0^t \sigma(X_r)dB_r = W_{\int_0^t \sigma^2(X_r)dr},\]

Thus, for establishing the uniqueness in law of solution $X$, one can consider the equation 
\begin{equation}\label{timechange} X_t = x_0 + \int_0^t f(X_r)dr + W_{\int_0^t \sigma^2(X_r)dr}
\end{equation}

where $W$ is a given standard Brownian motion. Thus equation
\eqref{timechange} is viewed as a time-changed equation of
\eqref{modelSDE}. The existence of a solution to \eqref{timechange} can be easily established by an application of
Schauder's fixed point theorem. For uniqueness, introduce 
\[ Z_t = \int_0^t \sigma^2(X_r)dr.\]
Since $|\sigma(\cdot)|\geq \lambda > 0$, $Z$ is a strictly increasing $C^1$ curve, $Y_s  = Z^{-1}(s)$ exists and is a $C^1$ curve with  
\[ \dot{Y}(Z(t)) = \frac{1}{\dot{Z}_t} = \frac{1}{\sigma^2(X_t)}\implies \dot{Y}_t = \frac{1}{\sigma^2(X_{Y_t})}.\]
Now look at $K_t = X_{Y_t}$. Then it follows that  
\[ K_t = x_0 + \int_0^t \frac{f(K_r)}{\sigma^2(K_r)}dr + W_t\]
Now appyling Davie's Theorem \cite{davie}, curve $K$ is uniquely determined by $W, f, x_0$ and $\sigma^2$. Also note that $X_t= K_{Z_t}$ and $\dot{Z}_t = \sigma^2(K_{Z_t})$ which implies   
\[ \int_0^{Z_t} \frac{dr}{\sigma^2(K_r)} = t.\]
Thus $Z_t$ is uniquely determined from $K_t$ and so is $X_t = K_{Z_t}$ implying the uniqueness of solution $X$. We note that changing \eqref{modelSDE} to \eqref{timechange} is a
probabilistic transformation and thus not suited to study path-by-path
uniqueness of \eqref{modelSDE}.
\end{itemize}
\end{remark}

{\bf Layout for the rest of the article:} In the next Section we prove
a rough path It\^o formula and a reduction to $\sigma \equiv 1$. This
is made precise in Proposition \ref{transformation}.  In Section
\ref{flowidentity}, we establish the flow Identity when $H =
\frac{1}{2}, d=1$, $\sigma(\cdot) \equiv 1$, $f$ compactly supported
in Proposition \ref{exp-identity}. With all the tools prepared we
prove the main results in Section \ref{pmrs}. We conclude the article
with an appendix section where we present preliminaries on rough paths
and provide a direct proof of path by path existence via an
application of Schauder's fixed point theorem.

{\bf Acknowledgements:}We would like to thank Leonid Mytnik for
various discussions during this project. S.A and A.S research were
supported in part by an ISF-UGC Grant. The second author was supported by NBHM (National Board for Higher Mathematics, Department of Atomic Energy, Government of India) Post Doctoral Fellowship.

\section{Rough It\^o Formula and Reduction to $\sigma \equiv 1$}\label{rough-Ito-formula}

The main result of this section is the following proposition.

\begin{proposition}\label{transformation} Assume (F1),(S1)-(S3). Then there exists a diffeomorphism
  $G:\R^d \to \R^d$ such that $DG(z) = \sigma(z)^{-1}$ and $(X, \sigma(X))$ is a solution to \eqref{modelRDE} if and only if $Z_t = G(X_t)$ satisfies 
\begin{equation} \label{unit-sigma} Z_t = G(x) + \int_0^t \tilde{f}(Z_r)dr + B^H_t,
\end{equation}
where $\tilde{f}: \R^d \rightarrow \R^d$ is given by $\tilde{f}(z)= DG(G^{-1}(z))f(G^{-1}(z))$ is a bounded continuous function.
\end{proposition}

 We shall imitate the argument mentioned in Remark \ref{s-rem}  for the proof of
 Proposition \ref{transformation}.  We first record some useful
 Lemmas. The following classical result on global inverse function
 theorem is well known.

\begin{lemma}[Hadamard Global Inverse Theorem] \label{hadamard}
If $G: \R^d \to \R^d$ is continuously differentiable map such that $|G(z)| \to \infty$ as $|z| \to \infty$ and $det(D G(z)) \neq 0$ for all $z$, then $G$ is a diffeomorphism of $\R^d$. 
\end{lemma}

\begin{proof} 
See Theorem $59$ of Chapter-$5$ in \cite{protter-book}.
\end{proof}
Our next lemma is an integration by parts formula.

\begin{lemma} \label{integration-by-parts} Let $\mathbf{B}^H$ be the
  geometric rough path defined above. Let $G: \R^d \rightarrow \R^d$
  be a $C^3$ function. Then for any $T >0$,
\begin{equation} G(B^H_T)= G(B^H_0) + \int_0^T DG(B^H_r)d\mathbf{B}^H_r.
\end{equation}

\end{lemma}

\begin{proof}
The equality \eqref{integration-by-parts} is an easy consequence of Proposition $2.5$ and Theorem $7.5$ in \cite{FH12}. But here we provide another direct argument. Note that since $(B^H, I)$ is a controlled rough path, $(DG(B^H),D^2G(B^H)I)$  is also a controlled rough path. For partitions $\mathcal{P}$ of $[0,T]$ with mesh size $|\mathcal{P}|$,
\begin{align*}\int_0^T DG(B^H_r)d\mathbf{B}^H_r & = \lim\limits_{|\mathcal{P}| \to 0} \sum_{[s,t]\in \mathcal{P}} DG(B^H_s)B^H_{s,t} + D^2G(B^H_{s})\mathbb{B}^H_{s,t} \\ 
& =  \lim\limits_{|\mathcal{P}| \to 0} \sum_{[s,t]\in \mathcal{P}} DG(B^H_s)B^H_{s,t} + \frac{1}{2}D^2G(B^H_{s}){B^H_{s,t}}^{\otimes 2} + \frac{1}{2} D^2G(B^H_{s})Anti(\mathbb{B}^H_{s,t}).
\end{align*}
Note that $D^2G$ is a symmetric operator and thus
$D^2G(B^H_{s})Anti(\mathbb{B}^H_{s,t}) = 0$. Also, since $G$ is $C^3$ function, by Taylor's expansion formula, 
\[G(B_t^H) - G(B_s^H) - DG(B^H_s)B^H_{s,t} - \frac{1}{2}D^2G(B^H_{s}){B^H_{s,t}}^{\otimes 2} = O (|B_{s,t}^H |^3)\]
Since $B^H$ is $\alpha$-H\"older for $\alpha > \frac{1}{3}$, it then follows that $$\lim\limits_{|\mathcal{P}| \to 0} \sum_{[s,t]\in \mathcal{P}} DG(B^H_s)B^H_{s,t} + \frac{1}{2}D^2G(B^H_{s}){B^H_{s,t}}^{\otimes 2} = G(B_T^H) - G(B_0^H).$$ This completes the proof. 
\end{proof}

Any solution $X$ of \eqref{modelRDE} can also be naturally lifted to a
rough path. Let $i : \mathbb{R}^d \to \mathbb{R}^d$ be the identity
map $i(x) = x$ and $\eta = i \otimes \sigma$. This means $\eta(x) = (
i \otimes \sigma)(x):\mathbb{R}^d \to \mathbb{R}^d\otimes
\mathbb{R}^d$ is a linear map defined by $\eta(x)y := x \otimes
(\sigma(x)y)$. Note that since $(X, \sigma(X))$ is a controlled rough
path, $(\eta(X), D\eta(X)\sigma(X))$ is also a controlled rough
path. Define $\mathbf{X} = (X, \mathbb{X}) :
            [0,T]\times [0,T] \rightarrow \mathbb{R}^d \oplus
            (\mathbb{R}^d\otimes \mathbb{R}^d)$ by
\begin{eqnarray}
&&  \mathbf{X}_{s,t}:= (X_{s,t},\mathbb{X}_{s,t}) \mbox{ with } X_{s,t}
  = X_t - X_s \mbox{ and } \nonumber\\ 
&& \mathbb{X}_{s,t} = \int_s^t (X_r - X_s)\otimes f(X_r)dr + \int_s^t  \eta(X_r)d\mathbf{B}^H_r - X_s \otimes \int_s^t
  \sigma(X_r)d\mathbf{B}^H_r. \label{rX}
\end{eqnarray}
Here the first integral is a Riemann integral and second and third integrals are rough integrals against rough path $\mathbf{B}^H$.

\begin{lemma}\label{geometric rough path} $\mathbf{X} = (X,
  \mathbb{X}) :      [0,T]\times [0,T] \rightarrow \mathbb{R}^d \oplus
            (\mathbb{R}^d\otimes \mathbb{R}^d)$ defined  in
  \eqref{rX} is an $\alpha$-geometric rough path. 
\end{lemma}

\begin{proof}
The verification of the Chen's relation 
\[ \mathbb{X}_{s,t} - \mathbb{X}_{s,u} - \mathbb{X}_{u,t} = X_{s,u}\otimes X_{u,t} \]
is immediate from the definition of $\mathbb{X}$. Also since $f$ is a bounded function and $(X, \sigma(X))$ is a controlled rough path, Theorem $4.10$ in \cite{FH12} easily implies 
\[ ||X||_{\alpha} + ||\mathbb{X}||_{2\alpha} < \infty.\]
We now verify that $\mathbf{X}$ is a geometric rough path. Without loss of generality we assume $X_s = 0$. By definition, 
\[ \mathbb{X}_{s,t} = \lim\limits_{|\mathcal{P}| \to 0 } S(\mathcal{P}),\]
where for partitions $\mathcal{P}$ of $[s,t]$, 
\begin{align*} S(\mathcal{P}) & := \sum_{[u,v] \in \mathcal{P}} X_u\otimes f(X_u) (v-u) + \eta(X_u)_{u,v} + D\eta(X_u)\sigma(X_u)\mathbb{B}^H_{u,v} \\ 
& = \sum_{[u,v] \in \mathcal{P}} X_u\otimes \{ f(X_u) (v-u) +
  \sigma(X_u)B^H_{u,v} + D\sigma(X_u)\sigma(X_u)\mathbb{B}^H_{u,v} \}
  \\ & \qquad \qquad +  \sum_{[u,v] \in
    \mathcal{P}}\{D\eta(X_u)\sigma(X_u) - X_u \otimes D\sigma(X_u)\sigma(X_u)\}\mathbb{B}^H_{u,v}.
\end{align*}

Again using Theorem $4.10$ in \cite{FH12} and continuity of $f$, we note that 
\begin{align*}
&f(X_u) (v-u) + \sigma(X_u)B^H_{u,v} + D\sigma(X_u)\sigma(X_u)\mathbb{B}^H_{u,v}\\
&= \int_u^v f(X_r)dr + \int_u ^v \sigma(X_r)d\mathbf{B}^H_r  + o(|v-u|)\\
&= X_v - X_u + o(|v-u|).
\end{align*}
By direct computations, 
\[ (\{D\eta(X_u)\sigma(X_u) - X_u \otimes D\sigma(X_u)\sigma(X_u)\}\mathbb{B}^H_{u,v})^{i,j} = \delta^{i,l}\sigma^{j,k}(X_u)\sigma^{l,a}(X_u){\mathbb{B}^H_{u,v}}^{a,k}.   \]
Here $v^{i,j}$ denotes the $(i,j)$ component of vector/tensor $v$, $\delta^{i,j}$ is the Kronecker delta function and we have used Einstein notation of summing over repeated indices. Thus, 
\[ S(\mathcal{P})^{i,j} = \sum_{[u,v]\in \mathcal{P}} X_u^iX_{u,v}^j + \sigma^{i,a}(X_u)\sigma^{j,k}(X_u){\mathbb{B}^H_{u,v}}^{a,k} + o(|v-u|).\]
Noting that $\mathbf{B}^H$ is a geometric rough path, we have
${\mathbb{B}^H_{u,v}}^{a,k} + {\mathbb{B}^H_{u,v}}^{k,a} ={B^H_{u,v}}^{a} {B^H_{u,v}}^k$. Thus, 
\[ S(\mathcal{P})^{i,j} + S(\mathcal{P})^{j,i} =  \sum_{[u,v]\in \mathcal{P}} X_u^iX_{u,v}^j + X_u^jX_{u,v}^i+ (\sigma(X_u){B^H}_{u,v})^i(\sigma(X_u){B^H}_{u,v})^j + o(|v-u|) \]
Finally, since $(X, \sigma(X))$ is a controlled rough path, 
\[ S(\mathcal{P})^{i,j} + S(\mathcal{P})^{j,i} =  \sum_{[u,v]\in \mathcal{P}} X_u^iX_{u,v}^j + X_u^jX_{u,v}^i+ X_{u,v}^iX_{u,v}^j + o(|v-u|)\]
and since $\mathbb{X}_{s,t} = \lim\limits_{|\mathcal{P}|\to 0 } S(\mathcal{P})$, $Sym(\mathbb{X}_{s,t}) = \frac{1}{2}X_{s,t}^{\otimes 2}$ which finishes the proof.
\end{proof}

We have constructed a well defined rough path $\mathbf{X}$ for
underlying path $X$. So we can make sense of the rough integral 
\[ \int_0^T F(X_r)d\mathbf{X}_r,\]
for any $F: \R^d \rightarrow M_{d\times d}(\R)$ which is continuously
twice differentiable. Also as $X$ is a $B^H$-controlled rough path,
$F(X)$ is also a $B^H$-controlled rough path and the rough integral
\[ \int_0^T F(X_r)d\mathbf{B}^H_r\]
is also well defined. The following Lemma establishes the relation
between two rough integrals.
\begin{lemma}[Rough It\^o Formula]\label{compatible}
If $F: \R^d \rightarrow $ is a $C^2$ $d\times d$ matrix valued function, then 
\[ \int_0^T F(X_r)d\mathbf{X}_r = \int_0^T F(X_r)f(X_r)dr + \int_0^T F(X_r)\sigma(X_r)d\mathbf{B}^H_r.\]
\end{lemma}

\begin{proof} Similar to the proof of Lemma \ref{geometric rough path}, it suffices to prove that 
\begin{align*}
F(X_u)X_{u,v} & + DF(X_u)\mathbb{X}_{u,v} \\
& - F(X_u)f(X_u)(v-u) - F(X_u)\sigma(X_u)B_{u,v}^H - \{F(X_u)\sigma(X_u)\}'\mathbb{B}^H_{u,v} = o(|v-u|).
\end{align*}

To this end, assume without loss of generality that $X_u =0$. Note that 
\[ X_{u,v} = f(X_u)(v-u) + \sigma(X_u)B_{u,v}^H + D\sigma(X_u)\sigma(X_u)\mathbb{B}_{u,v}^H + o (|v-u|), \]
and since $X_u =0$,
\[\mathbb{X}_{u,v} =  D\eta(X_u)\sigma(X_u) \mathbb{B}_{u,v}^H + o(|v-u|),\]
where $\eta = i \otimes \sigma$ as defined above. Also, if $F\sigma$ denotes the function $(F\sigma)(x) := F(x)\sigma(x)$, then it is easy to check that 
\[ D(F\sigma)(0) = F(0)D\sigma(0) + DF(0)D\eta(0)\]
The Gubinelli derivative $\{F(X_u)\sigma(X_u)\}'$ of the controlled rough path $F(X_u)\sigma(X_u)$ is thus given by 
\[ \{F(X_u)\sigma(X_u)\}' = F(X_u)D\sigma(X_u)\sigma(X_u) + DF(X_u)D\eta(X_u)\sigma(X_u).\] 
This completes the proof.
\end{proof}

\begin{proof}[Proof of Proposition \ref{transformation}]
Since $\sigma(x)$ satisfies $(S2)$, $det(\sigma(x)) \neq 0$ and $\sigma(x)^{-1}$ is a well defined invertible matrix. Let \[G(z) = \int_{\gamma} \sigma(x)^{-1}dx,\]
where $\gamma$ is any smooth curve joining $0$ and $z$. Since $(S3)$ holds, $G$ is a well defined function. Also, it can be easily verified that $DG(z)= \sigma(z)^{-1}$. Then it follows from Inverse Function Theorem that $G$ is a local diffeomorphism. In fact we claim that $G$ is a global diffeomorphism. In view of Lemma \ref{hadamard}, we just need to verify $|G(z)| \to \infty$ as $|z| \to \infty$. To this end, note that either $v^{\intercal}\sigma(x)v > 0$ for all $x$ and $ v \neq 0 $ or $v^{\intercal}\sigma(x)v < 0$ for all $x$ and $ v \neq 0 $. This easily follows from continuity of $\sigma$ and connectedness $\mathbb{R}^d \times \{\mathbb{R}^d\setminus {0}\}$ for $d\geq 2$. The case of $d=1$ can also be directly verified. Without loss of generality, we assume that $v^{\intercal}\sigma(x)v > 0$ for all $x$ and $ v \neq 0 $. From $(S2)$, we get 
\begin{equation}\label{lower-bound} v^{\intercal}\sigma(x)^{-1}v \geq \lambda ||\sigma(x)^{-1}v||^2.
\end{equation}
Since $\sigma$ is bounded, $||\sigma(x)^{-1}v|| \geq \frac{1}{||\sigma||_{\infty}}||v||$ and thus 
\[v^{\intercal}DG(x)v \geq \lambda' ||v||^2, \]
for some $\lambda'> 0$. For a fixed $v$, introduce the function $K: \mathbb{R} \to \mathbb{R}$ defined by $K(\theta) := v^{\intercal}G(\theta v)$. Then clearly $K'(\theta) = v^{\intercal}DG(\theta v)v \geq  \lambda' ||v||^2$. Integrating both sides, we get 
\[ v^{\intercal}\{G(v) - G(0)\} \geq \lambda' ||v||^2 \]
From Cauchy-Schwarz inequality, $ |v^{\intercal}\{G(v) - G(0)\}| \leq ||G(v) - G(0)||||v||$. This implies $||G(v) - G(0)|| \geq \lambda' ||v||$. In particular $||G(v)|| \to \infty$ as $||v|| \to \infty$. This proves that $G$ is in fact a global diffeomorphism. \\
Also, again from Cauchy-Schwarz inequality, $v^{\intercal}\sigma(x)^{-1}v \leq ||\sigma(x)^{-1}v ||||v ||$ and it follows from \eqref{lower-bound} that $||\sigma(x)^{-1}v|| \leq \frac{1}{\lambda}||v||$ which proves that $ DG = \sigma^{-1}$ is a bounded function. Thus in particular, $\tilde{f}$ is a bounded continuous function. \\
Finally, let $(X, \sigma(X))$ be a solution to \eqref{modelRDE}. $X$ can be lifted to a geometric rough path $\mathbf{X}$ given by Lemma \ref{geometric rough path}. Using Lemma \ref{integration-by-parts}, \[ G(X_t) = G(x_0 ) + \int_0^t DG(X_r)d\mathbf{X}_r.\]
Since $DG(x) = \sigma(x)^{-1}$, using Lemma \ref{compatible}, 
\begin{align*} \int_0^t DG(X_r)d\mathbf{X}_r & = \int_0^t DG(X_r)f(X_r)dr + \int_0^t DG(X_r)\sigma(X_r)d\mathbf{B}^H_r \\
& = \int_0^t DG(X_r)f(X_r)dr + \mathbf{B}^H_t
\end{align*}
Thus, $Z_t = G(X_t)$ satisfies the equation \eqref{unit-sigma}.
Conversely, if $Z$ is a solution to \eqref{unit-sigma}, applying Lemma \ref{integration-by-parts} and \ref{compatible} to $ G^{-1}(Z_t)$ proves that $X_t = G^{-1}(Z_t)$ solves \eqref{modelRDE}.
\end{proof}

\section{Flow for $H= \frac{1}{2}, d=1$, $\sigma \equiv 1$, $f$ H\"older and compactly supported} \label{flowidentity}

In this section we restrict ourselves to $d=1, H = \frac{1}{2}$ case. Let $\sigma\equiv 1$ and $f$ be a compactly supported $\theta$-H\"older function for $\theta \in (0,1)$ in
\eqref{modelRDE}. Then \eqref{modelRDE} reduces to
\begin{equation} \label{modelSDE1}
  X_t = x + \int_0^t f(X_r) dr + B_t.
\end{equation}

We will derive properties of the flow associated with equation \eqref{modelSDE1}. Existence of a flow $\phi(s,t,x)$ was shown in \cite{Flandoli-Fedrizzi}. The definition of a flow in \cite{Flandoli-Fedrizzi} is somewhat different from Definition \ref{just-flow} here, see Definition $5.1$ in \cite{Flandoli-Fedrizzi}. We consider the flow $\phi(s,t,x)$ constructed in \cite{Flandoli-Fedrizzi} according to their definition. Thus the random field $\phi(s,t,x)$ satisfies:
\[ \mathbb{P}[ \phi(s,t,x) = \phi(u,t, \phi(s,u,x)) \hspace{2mm}\mbox{and}\hspace{2mm} \phi(s,s,x) = x \hspace{2mm} \mbox{for all $s \leq u \leq t $ and $x \in \mathbb{R}$}] = 1\] 
and for each fixed $s\in [0,T]$, $\phi(s,t,x)$ is almost surely jointly continuous in $(t,x)$, adapted to the filtration of Brownian motion $B_t^s = B_t - B_s$ and satisfies 
\[ \phi(s,t,x) = x + \int_s^t f(\phi(s,r,x))dr + B_t - B_s.\]

We will prove that field $\phi$ has a continuous modification $\psi$
which satisfies the properties of flow as per Definition
\ref{just-flow} with $\sigma \equiv 1$. For fixed $s\geq 0$ and $x, y
\in \R$, define for $s\leq t <T$ and $u \in [0,1]$,
\[ Z_t^u := u \phi(s, t, x) + (1-u) \phi(s,t , y),\]
\[ I_t^u := F(Z_t^u) - F(Z_s^u) - \int_s^t f(Z_r^u)(uf(\phi(s,r,x)) + (1-u)f(\phi(s,r,y)))dr \]
where $F$ is the antiderivative of function $f$, i.e. $F' = f$, 
\[ J_t^u := \int_s^tf(Z_r^u)dB_r,\]
and  the random field $J$ by
\[ J(s,t, x, y) := \int_0^1 J_t^udu = \int_0^1 \int_s^t f(u\phi(s,r,x) + (1-u)\phi(s,r, y))dB_rdu.\]
Our main result in this section is the following.
\begin{proposition} \label{exp-identity} Let $\phi, I,J$ be as given above.
  
  \begin{enumerate} 
    \item[(a)]For each fixed $s,x, y$, almost surely for all $s \leq t <T$, 
\begin{equation}\label{derivative-identity}\phi(s,t,x) - \phi(s,t,y) = (x-y) \exp\biggl[2 \int_0^1 \{I_t^u - J_t^u\} du\biggr] 
\end{equation}
\item[(b)] For each $p \geq 2 $ and $T>0$, there exists a constant $ C_1 = C_1(p, f, T)$ depending only on $p,$ $f$ and $T$ such that for all $0\leq s \leq t \leq T$, $0\leq \tilde{s}\leq \tilde{t}\leq T$ and $x,\tilde{x} \in \R$, 
\begin{equation} \label{Lp-bound}
 \mathbb{E}[| \phi(s,t,x) - \phi(\tilde{s}, \tilde{t}, \tilde{x}) |^p] \leq  C_1 ( |x-\tilde{x}|^p +  |t-\tilde{t}|^{\frac{p}{2}} + |s-\tilde{s}|^{\frac{p}{2}})\end{equation}
\item[(c)] For each $p \geq \frac{2}{\theta} $ and $T>0$, there exists a constant $ C_2 = C_2(p, f, T)$ depending only on $p,$ $f$ and $T$ such that for all $0\leq s \leq t \leq T$, $0\leq \tilde{s}\leq \tilde{t}\leq T$ and $x,y,\tilde{x},\tilde{y} \in \R$,
\begin{equation} \label{stochastic-integral-continuous}\mathbb{E}[|J(s,t, x, y) - J(\tilde{s}, \tilde{t}, \tilde{x}, \tilde{y})|^p] \leq C_2( |s-\tilde{s}|^{\frac{p\theta}{2}} + |t-\tilde{t}|^{\frac{p\theta}{2}} + |x-\tilde{x}|^{p\theta} + |y-\tilde{y}|^{p\theta}).\end{equation}
\item[(d)]$\phi$  has a  modification which is almost surely $\eta$-H\"older continuous jointly in $(s,t,x)$ for any $\eta < \frac{1}{2}$ and $J$ has a modification which is almost sure $\eta$-H\"older continuous for any $\eta < \frac{\theta}{2}.$
  
  \end{enumerate}
\end{proposition}

\begin{proof}[Proof of Proposition \ref{exp-identity}(a)]

First we consider the case when function $f$ is continuously differentiable. Defining $S_t = \phi(s,t,x) - \phi(s,t,y)$, note that 
\begin{align*} S_t & = x-y + \int_s^t \{f(\phi(s,r,x)) - f(\phi(s,r,y))\}dr \\
& = x-y + \int_s^t \int_0^1 f'(u \phi(s,r,x) + (1-u)\phi(s,r, y)) (\phi(s,r,x) - \phi(s,r,y))du dr \\
&= x-y + \int_s^t \biggl\{\int_0^1 f'(Z_r^u)du\biggr\} S_rdr
\end{align*}
Solving the above linear differential equation, we get 
\begin{align*} S_t &= (x-y) \exp\biggl[ \int_s^t \int_0^1 f'(Z_r^u)dudr \biggr] \\ 
&= (x-y) \exp\biggl[  \int_0^1 \int_s^t  f'(Z_r^u)dr du \biggr]
\end{align*}
Now note that $Z_t^u$ is a semimartingale with the quadratic variation $[Z^u]_t = t $ and using It\^o formula, 
\[ \int_s^t  f'(Z_r^u)dr = 2 \biggl\{ F(Z_t^u) - F(Z_s^u) - \int_s^t f(Z_r^u)dZ_r^u\biggr\} = 2\{I_t^u - J_t^u\},\]
which establishes the claim when $f$ is continuously differentiable. Now for the general case when $f$ is just a continuous compactly supported function, we consider a sequence of continuously differentiable functions $f^n$ converging uniformly to $f$ as $n \to \infty$. Using Theorem $2.1$ in \cite{KR}, we know that $\phi(s,t,x)$ is the unique strong solution to the equation 
\[ \phi(s,t, x) = x + \int_s^t f(\phi(s,r,x))dr + B_t - B_s.\]
Let $\phi^n(s,t,x)$ denote the solution of above equation when $f$ is replaced by $f^n$. Then we easily conclude that almost surely $\phi^n(s,t,x)$ converge to $\phi(s,t,x)$ uniformly in $t$ as $n \to \infty$. It then easily follows that $\int_0^1 I_t^{u,n}du$ and $\int_0^1 J_t^{u,n}du$ converge to $\int_0^1 I_t^{u}du$ and $\int_0^1 J_t^{u}du$ respectively as $n\to \infty$ and this completes the proof.

\end{proof}

\begin{proof}[Proof of Proposition \ref{exp-identity}(b)] W.l.o.g. we can assume $s \leq \tilde{s}$. If $t \in [s, \tilde{s}]$, we use the bound   
  \begin{eqnarray*}
    \lefteqn{|\phi(s,t,x) - \phi(\tilde{s}, \tilde{t}, \tilde{x})|}\\
    &=&\bigl| x - \tilde{x} + \int_s^t f(\phi(s,r,x))dr - \int_{\tilde{s}}^{\tilde{t}}f(\phi(\tilde{s}, r, \tilde{x}))dr + B_t - B_s - (B_{\tilde{t}} - B_{\tilde{s}}) \bigr| \\
&\leq& |x - \tilde{x}| +  ||f||_{\infty}( |t-s| + |\tilde{t} - \tilde{s}| ) + |B_t - B_{\tilde{t}}| + |B_{s} - B_{\tilde{s}}|.
\end{eqnarray*}
Since in this case $|t-s| \leq |s- \tilde{s}|$ and $|\tilde{t} - \tilde{s}| \leq |\tilde{t} - t|$, using the standard moment bounds on the increments of Brownian motion, the claim follows.\\
When $t \geq \tilde{s}$, using the flow property, $\phi(s,t,x) - \phi(\tilde{s}, \tilde{t}, \tilde{x}) = \phi(\tilde{s}, t, \phi(s, \tilde{s}, x)) - \phi(\tilde{s}, \tilde{t}, \tilde{x})$. Again, since \[|\phi(\tilde{s}, \tilde{t}, \tilde{x}) - \phi(\tilde{s}, t, \tilde{x})| \leq ||f||_{\infty}|t-\tilde{t}| + |B_{t} - B_{\tilde{t}}|,\] it is enough to get the desired moment estimates on $ |\phi(\tilde{s}, t, \phi(s, \tilde{s}, x)) - \phi(\tilde{s}, t, \tilde{x})|.$ Using Proposition \ref{exp-identity} (a), 
\[ |\phi(\tilde{s}, t, \phi(s, \tilde{s}, x)) - \phi(\tilde{s}, t, \tilde{x})| = | \phi(s, \tilde{s}, x)-\tilde{x}| \exp\biggl[2  \int_0^1 \{I_t^u - J_t^u\}du\biggr].\]
Note that $\phi(s, \tilde{s}, x)$ is independent of $I_t^u$ and $J_t^u$ $($please note that we have replaced the parameter $s,x,y$ to $\tilde{s}, \tilde{x}, \phi(s,\tilde{s},x))$ respectively in the definition of $I_t^u$ and $J_t^u$ here$)$ and 
\[ | \phi(s, \tilde{s}, x)-\tilde{x}| \leq |x-\tilde{x}| + ||f||_{\infty}|s-\tilde{s}| + |B_s-B_{\tilde{s}}|.\]
Thus it suffices to prove that 
\begin{equation}\label{exp-martingale-bound} \mathbb{E}\biggl\{\exp\biggl[2 p \int_0^1 \{I_t^u - J_t^u\}du\biggr]\biggr\} \leq C(p, ||f||_{\infty},T).
\end{equation}
To this end, note that 
\[ |I_t^u|   \leq 2 ||f||_{\infty}^2 |t-\tilde{s}| + ||f||_{\infty}|B_t -B_{\tilde{s}}|\]
and $J_t^u$ is a martingale with $[J^u]_t \leq ||f||_{\infty}^2 (t - \tilde{s})$. Also by Dambis-Dubins-Schwarz martingale embedding theorem, $J_t^u = \tilde{B}_{[J^u]_t}$ for some another Brownian motion $\tilde{B}$. Finally, using Fernique Theorem gives us the bound \eqref{exp-martingale-bound} and this concludes the proof.
\end{proof}

\begin{proof}[Proof of Proposition \ref{exp-identity}(c)]W.l.o.g. we assume $s \leq \tilde{s}$. If $t \in [s, \tilde{s}]$, then using Fubini Theorem for stochastic integrals and Burkholder-Davis-Gundy inequality, 
\begin{align*} \mathbb{E}[|J(s,t, x, y)|^p] &= \mathbb{E}\biggl| \int_s^t\int_0^1 f(u\phi(s,r,x) + (1-u)\phi(s,r, y))dudB_r\biggr|^p \\ & \leq C(p) \mathbb{E}\biggl( \int_s^t\biggl\{\int_0^1 f(u\phi(s,r,x) + (1-u)\phi(s,r, y))du\biggr\}^2dr\biggr)^\frac{p}{2} \\
  & \leq C(p, ||f||_{\infty}) |t-s|^{\frac{p}{2}}\\
& \leq C(p, ||f||_{\infty}) |s-\tilde{s}|^{\frac{p}{2}}.
\end{align*}
Similarly, $\mathbb{E}[|J(\tilde{s}, \tilde{t}, \tilde{x}, \tilde{y})|^p]  \leq  C(p, ||f||_{\infty}) |t-\tilde{t}|^{\frac{p}{2}}$ which gives us \[ \mathbb{E}[|J(s,t, x, y) - J(\tilde{s}, \tilde{t}, \tilde{x}, \tilde{y})|^p] \leq C(p, ||f||_{\infty})( |s-\tilde{s}|^{\frac{p}{2}}+ |t-\tilde{t}|^{\frac{p}{2}} ).\]
If $t \geq\tilde{s}$, then note that 
\begin{eqnarray*}
  J(s,t,x, y) &=& \int_s^{\tilde{s}} \int_0^1 f(u\phi(s,r,x)  + (1-u)\phi(s,r, y))dudB_r \\ && \hspace{0.5in} + \int_{\tilde{s}}^t\int_0^1 f(u\phi(s,r,x) + (1-u)\phi(s,r, y))dudB_r\end{eqnarray*}
and
\begin{eqnarray*}
 J(\tilde{s}, \tilde{t}, \tilde{x}, \tilde{y}) &=& \int_t^{\tilde{t}} \int_0^1 f(u\phi(\tilde{s},r,\tilde{x}) + (1-u)\phi(\tilde{s},r, \tilde{y}))dudB_r \\ && \hspace{0.5in} + \int_{\tilde{s}}^{t} \int_0^1 f(u\phi(\tilde{s},r,\tilde{x}) + (1-u)\phi(\tilde{s},r, \tilde{y}))dudB_r.
  \end{eqnarray*}
Thus using the previous step, 
\begin{align*}
\mathbb{E}[|J(s,t, x, y) - J(\tilde{s}, \tilde{t}, \tilde{x}, \tilde{y})|^p] & \leq C(p, ||f||_{\infty})( |s-\tilde{s}|^{\frac{p}{2}}+ |t-\tilde{t}|^{\frac{p}{2}} )  + \Theta(s,t, x, y, \tilde{s},\tilde{x},\tilde{y})
\end{align*}
where $ \Theta(s,t, x, y, \tilde{s},\tilde{x},\tilde{y})$ is defined by
\begin{eqnarray*}
  \lefteqn{\Theta(s,t, x, y, \tilde{s},\tilde{x},\tilde{y})}\\
  &= &\mathbb{E}\biggl[ \biggl|\int_{\tilde{s}}^t\int_0^1 \{ f(u\phi(s,r,x) + (1-u)\phi(s,r, y)) -  f(u\phi(\tilde{s},r,\tilde{x}) + (1-u)\phi(\tilde{s},r, \tilde{y}))\}dudB_r\biggr|^p\biggr]  \end{eqnarray*}
Again using Burkholder-Davis-Gundy inequality and \eqref{Lp-bound}, 
\begin{eqnarray*}\lefteqn{\Theta(s,t, x, y, \tilde{s},\tilde{x},\tilde{y})}\\&& \leq C(p,T) \times \\ && \times \int_{\tilde{s}}^t\int_0^1 \mathbb{E}[| f(u\phi(s,r,x) + (1-u)\phi(s,r, y)) -  f(u\phi(\tilde{s},r,\tilde{x}) + (1-u)\phi(\tilde{s},r, \tilde{y}))|^p]dudr\\
&& \leq C(p,\theta, ||f||_{\theta}, T) \int_0^T\{ \mathbb{E}[|\phi(s,r,x) - \phi(\tilde{s},r,\tilde{x}|^{p\theta }] + \mathbb{E}[|\phi(s,r, y)- \phi(\tilde{s},r, \tilde{y})|^{p\theta}]\}dr \\
&&\leq C(p,\theta, ||f||_{\theta}, T)  (| s-\tilde{s}|^{\frac{p\theta}{2}} + |x-\tilde{x}|^{p\theta} + |y-\tilde{y}|^{p\theta} )
\end{eqnarray*}

\end{proof}

\begin{proof}[Proof of Proposition \ref{exp-identity}(d)] Result follows from (b), (c) and application of the Kolmogorov Continuity Theorem.
\end{proof}

\section{Proof of Main Results} \label{pmrs}

\subsection{Proof of Theorem \ref{gen-catellier-gubinelli}}

\begin{proof}
 Path by path existence of solutions for (\ref{unit-sigma}) follows easily via an application of  Schauder's fixed point Theorem. For completeness sake, we show this in Section \ref{existence}. The reader may also refer to  \cite{CG} for $\frac{1}{3} <H < \frac{1}{2}$ and \cite{davie} for $H= \frac{1}{2}$. So using Proposition \ref{transformation} path by path existence of solutions to (\ref{modelRDE}) follows.
  
  Let $X$ be a solution to \eqref{modelRDE}, then by Proposition \ref{transformation}, $Z_t = G(X_t)$ solves \eqref{unit-sigma}. From Theorem $1.8$ in \cite{CG} (for $\frac{1}{3} <H < \frac{1}{2}$) and Theorem $1.1$ in \cite{davie} (for $H= \frac{1}{2}$), it follows that $Z$ is the path-by-path unique solution to equation \eqref{unit-sigma}. Finally since $G$ is bijective, \eqref{modelRDE} has the path-by-path unique solution. Since $G$ is determined by $f$ and $\sigma$, the null set obtained for path by path uniqueness will depend on $x_0,f,$ and $\sigma$.  
\end{proof}

\subsection{Proof of Theorem \ref{smooth-flow} and Theorem \ref{pbypins}}

\subsubsection{$\sigma\equiv 1$ and $f$ compactly supported}
In this subsection we prove Theorem \ref{smooth-flow} and Theorem \ref{pbypins} under the assumption $\sigma \equiv 1$ and $f$ compactly supported.  

\begin{proof}[Proof of Theorem \ref{smooth-flow}]

 (a) Let $\phi$ be the flow constructed in \cite{Flandoli-Fedrizzi}. From Proposition \ref{exp-identity} (d) we know that $\phi$ has a continuous modification, call it $\psi$. Then it can be easily checked that $\psi$ satisfies the conditions in Definition \ref{just-flow}. 

 (b) From Proposition \ref{exp-identity} (d) we know that $J$ has a continuous modification, call it $\tilde{J}$. Let $$\tilde{Z}^u(s,t,x,y) = u\psi(s,t,x) + (1-u)\psi(s,t,y)$$ and 
\begin{eqnarray*}
  \lefteqn{\tilde{I}^u(s,t,x,y) }\\
  &=&F(\tilde{Z}^u(s,t,x,y))-F(\tilde{Z}^u(s,s,x,y))  \\
& -&\int_s^t f(\tilde{Z}^u(s,r,x,y))(uf(\psi(s,r,x)) + (1-u)f(\psi(s,r,y)))dr,
\end{eqnarray*}
where $F' = f$. Since $\psi$ is a continuous modification of $\phi$, the function $\tilde{I}$ defined by 
\[\tilde{I}(s,t,x,y) := \int_0^1\tilde{I}^u(s,t,x,y)du \] is almost surely jointly continuous in $s,t,x,y$.  Since $\psi, \tilde{I}$ and $\tilde{J}$ are continuous modifications of $\phi, I, J$ respectively, Proposition \ref{exp-identity} (a) implies that almost surely for all $s,t,x,y$,
\[ \psi(s,t,x) - \psi(s,t,y) = (x-y)\exp[2\{\tilde{I}(s,t,x,y) - \tilde{J}(s,t,x,y)\}].\]
Thus we conclude that $\psi$ is almost surely differentiable with 
\[ D\psi(s,t,x) = \exp[2\{\tilde{I}(s,t,x,x) - \tilde{J}(s,t,x,x)\}],\]
which proves the smoothness of the flow. Now, from Proposition \ref{exp-identity} (d), $\psi$ is almost surely jointly $\eta$-H\"older continuous for any $\eta< \frac{1}{2}$ and $\tilde{J}$ is almost surely jointly $\eta$-H\"older continuous for any $\eta < \frac{\theta}{2}$. This implies $\{\tilde{I}(s,t,x,x) - \tilde{J}(s,t,x,x)\}$ is jointly $\eta$-H\"older continuous for any $\eta < \frac{\theta}{2}$ which concludes the proof.

\end{proof}

\begin{proof}[Proof of Corollary \ref{likeflandoli}]
Note that $|\psi(s,t,x)| \to \infty$ as $x \to \infty$. Since $|D\psi(s,t,x)| \neq 0$, the claim follows from Hadamard global inverse theorem. 
\end{proof} 

\begin{proof}[Proof of Theorem \ref{pbypins}]
We consider the null set $\mathcal{N}$ obtained from Theorem \ref{smooth-flow}. Thus, in particular for all $\omega \notin \mathcal{N}$ and for all $T, R > 0$, \[||\psi||_{1,R} := \sup_{0\leq s \leq t\leq T} \sup_{|x|,|y| \leq R}\frac{|\psi(s,t,x) - \psi(s,t,y)|}{|x-y|} < \infty.\] 

Now let $X_t$ be a solution to equation \eqref{modelSDE1}. Fix $t \in [0,T]$ and for $r \in [0,t]$, define $L_r = \psi(r,t,X_r) - \psi(0,t,x)$. Now, for $u\leq r$, 
\begin{align*} |L_r - L_u| &= |\psi(r, t, X_r) - \psi(u, t, X_u)| \\
&= |\psi(r, t, X_r) -  \psi(r,t,\psi(u,r, X_u))|\\
&\leq ||\psi||_{1,R}|X_r - \psi(u,r, X_u)|,
\end{align*}
for some large enough $R$.
Note that  \[ |X_r  -\psi(u, r, X_u) | \leq \int_u^r |f(X_{\theta}) -  f(\psi(u,\theta, X_u))|d\theta,\]
which implies \[ \frac{|L_r - L_u|}{|r-u|} \leq ||\psi||_{1,R} \frac{1}{r-u}\int_u^r |f(X_{\theta}) -  f(\psi(u,\theta, X_u))|d\theta.\]
Taking $r \downarrow u$ or $u \uparrow r$ implies $\frac{dL_r}{dr} = 0$ for all $r \in [0,t]$ and thus $L_t = L_0$. Since $L_0 = 0$, we conclude that  $X_t =  \psi(0,t,x)$ which establishes the uniqueness of $X$.
\end{proof}

\subsubsection{General $\sigma$ and $f$}

 If $\sigma \equiv 1$ and $f$ has non compact support, we choose a sequence of bounded continuous functions $f^n$ such that $$
 f^n(x) = \begin{cases}  f(x) & \mbox{ for }  |x| \leq n\cr
   0 & \mbox{ for } |x| > 2n. \end{cases} $$ Let $\psi^n$ be the flow,  corresponding to $f^n$,  constructed in the proof of Theorem \ref{smooth-flow} in the  previous section. Each $\psi^n$ is continuously differentiable for $\omega \notin \mathcal{N}_n$ and null sets $\mathcal{N}_n$ do not depend on $x$ in the proof of Theorem \ref{pbypins} in the previous section. 

 Then for $\mathcal{N} = \cup_n \mathcal{N}_n$,  $\omega \notin \mathcal{N}$ and $x \in \R$, define $\psi(s,t,x) = \psi^n(s,t,x)$ by choosing $n$ large enough (depending on $\omega, x$ and $t$).

  $\psi$ defined above is a well defined function (courtesy Theorem \ref{pbypins} in the previous section). Further $\psi$ is the flow corresponding to function $f$, i.e. satisfies Definition \ref{just-flow} for function $f$. Clearly, $\psi$ is automatically continuously differentiable.

 The general $\sigma$ case follows from direct application of Proposition \ref{transformation} if the rough path $\mathbf{B} = \mathbf{B}^{Strat}$ is chosen. If the rough path $\mathbf{B} = \mathbf{B}^{Ito}$ is chosen, by \eqref{correction-term}, one can transform equation \eqref{modelRDE} back to the choice of $\mathbf{B} = \mathbf{B}^{Strat}$ by changing the function $f$ to $\tilde{f}$ given by $\tilde{f}(x) = f(x) - \frac{1}{2}D\sigma \sigma$. This completes the proof of Theorem \ref{smooth-flow}.

 The proof of Theorem \ref{pbypins} follows as in the previous section.

\begin{appendices}
  \section{Preliminaries on Rough Paths. } \label{a-prps}

  We shall state some basic  preliminaries from rough path theory. We refer the
reader to \cite{FH12} and \cite{FV10} for detailed exposition of the
same. Let $T>0, d \geq 1$ be fixed and we shall refer to $W:[0,T]
\rightarrow \R^d$ as a path. We will use the notation of $$W_{s,t} :=
W_t - W_s$$ for $s,t \in [0,T]$ a path $W:[0,T] \to \mathbb{R}^d$.
For $\alpha \in (0,1)$, we shall call a path $W:[0,T] \to
\mathbb{R}^d$ a $\alpha$-H\"older path if the path is H\"older
continuous with exponent $\alpha$. For vector spaces $V$ and $W$,
$V\otimes W$ denotes the tensor product of these vector spaces. We
first define the notion of rough paths.

\begin{definition}[Rough paths] \label{rough-path}
   Given   $\alpha \in (0, 1)$ and $W:[0,T] \to
     \mathbb{R}^d$ a $\alpha$-H\"older path, a $\alpha$-rough path
     $\mathbf{W}$ is a pair of maps $\mathbf{W} = (W, \mathbb{W}) :
            [0,T]\times [0,T] \rightarrow \mathbb{R}^d \oplus
            (\mathbb{R}^d\otimes \mathbb{R}^d)$ such that:
\begin{itemize}
\item For $s,t \in [0,T]$, $W_{s,t} = W_t - W_s $
\item For any $s, u, t \in [0,T]$ 
\begin{equation}\label{chen}
\mbox{Chen's relation:} \hspace{4mm}\mathbb{W}_{s,t} - \mathbb{W}_{s,u} - \mathbb{W}_{u,t} = W_{s,u}\otimes W_{u,t}.
\end{equation}

\item $W$ and $\mathbb{W}$ has $\alpha$ and $2 \alpha$ H\"older regularity respectively, i.e. 
\begin{equation}\label{analytical condition} ||W||_{\alpha} := \sup_{s,t \in[0,T], s \neq t} \frac{|W_{s,t}|}{|t-s|^{\alpha}} < \infty, \hspace{4mm} ||\mathbb{W}||_{2\alpha} := \sup_{s,t \in[0,T], s \neq t} \frac{|\mathbb{W}_{s,t}|}{|t-s|^{2\alpha}} < \infty 
\end{equation}

\end{itemize}
Further, a rough path $\mathbf{W}$ is called a geometric rough path if in addition for all $s,t \in [0,T]$
\begin{equation}\label{GRP} Sym(\mathbb{W}_{s,t}) = \frac{1}{2}W_{s,t}\otimes W_{s,t},
\end{equation}
where $Sym(M)$ denotes the symmetric part of matrix $M$. Lastly, for $s \in [0,T)$, a $\alpha$-H\"older path $W:[0,T]\to \mathbb{R}^d$ is said to be ``rough at time $s$" if for all $\phi \in (\mathbb{R}^d)^{*}$, $\phi \neq 0$, 
\[ \limsup_{t \downarrow s} \frac{|\phi(W_{s,t})|}{|t-s|^{2\alpha}} = +\infty\] 
$W$ is called truly-rough if $W$ is rough on some dense subset of $[0,T]$.

\end{definition}
For our path by path construction we need the notion of rough integral which are  constructed using controlled rough paths.
\begin{definition}[Controlled rough paths] \label{CRP} Given an $\alpha$-rough path $\mathbf{W}$, a pair of continuos paths $(Y, Y')$, with $Y$ taking value in some $\mathbb{R}^m$ and $Y'$ in $\mathbb{R}^{m \times d }$, is called a $W$-controlled rough path if  
\begin{itemize}
\item $||Y||_{\alpha} + ||Y'||_{\alpha} < \infty.$
\item The object $ R^Y$ defined by $R^Y_{s,t} := Y_{s,t} - Y_s'X_{s,t}$ has $2 \alpha$ regularity, i.e. \[||R^Y||_{2\alpha} := \sup_{s,t }\frac{|R^Y_{s,t}|}{|t-s|^{2\alpha}} < \infty.\]
  
\end{itemize} 
The path $Y'$ is called the Gubinelli derivative of $Y$ with respect to $X$.
\end{definition}
We denote $\mathcal{C}^{\alpha}_W(\mathbb{R}^m)$ to be Banach space of all controlled rough paths $(Y, Y')$ taking values in  $\mathbb{R}^m \oplus \mathbb{R}^{m\times d}$ with the norm defined as 
\[ ||(Y, Y')||_{\mathcal{C}^{\alpha}_W} := |Y_0| + |Y_0'| + ||(Y,Y')||_{W,\alpha}\]
where \[
 ||(Y,Y')||_{W,\alpha} := ||Y'||_{\alpha} + ||R^Y||_{2\alpha}
 \]
 The concept of true-roughness of rough paths, defined above,  was introduced in
\cite{FS1}. It enables us to
define controlled rough paths in a unique and unambiguous manner. That is, if $(Y,Y'), (\tilde{Y}, \tilde{Y}') \in C_W^{\alpha}(\mathbb{R}^d)$ are two controlled rough paths and $W$ is truly rough, then 
\begin{equation} \label{doob-meyer-thm} Y\equiv \tilde{Y} \implies Y' \equiv \tilde{Y}'.\end{equation}

The rough integral is constructed as follows. Let $\mathbf{W}$ be a rough path and $(Y, Y')$ be a controlled rough path taking values in $(\mathbb{R}^{m \times d}, \mathbb{R}^{m \times d \times d} )$. Then the rough integral, 
\[ \int_0^T Y_rd\mathbf{W}_r := \lim\limits_{|\mathcal{P}| \to 0 } \sum_{[s,t] \in \mathcal{P}} Y_sW_{s,t} + Y_s'\mathbb{W}_{s,t}\]

exists, where for a partition $\mathcal{P}$ of [0,T], $[s,t] \in
\mathcal{P}$ denotes a subinterval of $\mathcal{P}$ and
$|\mathcal{P}|$ denotes its mesh size. Furthermore there exists a
constant $C$ depending only on $\alpha$ such that for all $s, t$,
\begin{equation}\label{rough inequality} \biggl|\int_s ^t Y_rd\mathbf{W}_r - Y_sW_{s,t} - Y_s'\mathbb{W}_{s,t} \biggr| \leq C ( ||W||_{\alpha}||R^Y||_{2\alpha} + ||\mathbb{W}||_{2\alpha}||Y'||_{\alpha}) |t-s|^{3\alpha} 
\end{equation}

{\bf Choice of Rough Paths for $B$:} When $H=\frac{1}{2}$, $B^H$
corresponds to standard Brownian motion $B$ (where $B$ is an
abuse of notation for $B^{\frac{1}{2}}$.). Almost surely, such paths
can also be lifted to a pair $(B,\mathbb{B})$ making it into a
rough path. Thus for certain adapted processes $Y$, we can give appropriate
rough path interpretations for the Ito-Integral $ \int_0^T Y_rdB_r$
and the Stratonovich integral $\int_0^T Y_r\circ dB_r.$ We refer the
reader to \cite{FH12} for details. In short, there exists a null set
$\mathcal{N}$ such that for $\omega \notin \mathcal{N}$,
$\mathbf{B}^{Ito}(\omega)$ and $\mathbf{B}^{Strat}(\omega)$ defined by
\[\mathbf{B}_{s,t}^{Ito}(\omega):= \biggl(B_t(\omega) - B_s(\omega), \int_s^t B_{s,r}\otimes dB_r(\omega)\biggr)\]
\[\mathbf{B}_{s,t}^{Strat}(\omega):= \biggl(B_t(\omega) - B_s(\omega), \int_s^t B_{s,r}\otimes \circ dB_r(\omega)\biggr)\]
are well defined $\alpha$-rough paths for all $\frac{1}{3}<\alpha < \frac{1}{2}$. The rough path $\mathbf{B}^{Strat}$ is in fact a geometric rough path. Furthermore, for any $B$-controlled rough path $(Y,Y')$
\begin{equation}\label{correction-term} \int_0^T (Y_r, Y_r')d\mathbf{B}_r^{Strat} = \int_0^T (Y_r, Y_r')d\mathbf{B}_r^{Ito} + \frac{1}{2}\int_0^T Y_r'dr.
\end{equation}

If $(Y(\omega), Y'(\omega))$ are random controlled rough paths such that $Y$ and $Y'$ are adapted, then almost surely,
\[ \int_0^T (Y_r, Y_r')d\mathbf{B}_r^{Ito} = \int_0^T Y_rdB_r.\]
Furthermore, if the quadratic covariation $[Y,B]$ exists, then almost
surely \[ \int_0^T (Y_r, Y_r')d\mathbf{B}_r^{Strat} = \int_0^T
Y_r\circ dB_r.\]
 
  \section{Existence of Solutions}
\label{existence}
 
In this section we establish existence of a solution to \eqref{modelRDE} directlty under the assumption that $f$ is bounded continuous and $\sigma$ is a $C^3_b$ function. The existence of path by path solution for (\ref{unit-sigma}) immediately follows.

We will rely on Schauder's fixed point theorem to prove existence of a solution to \eqref{modelRDE}. Following a trivial
remark that $\alpha$-rough paths are also $\beta$-rough path for any
$\beta \in (\frac{1}{3}, \alpha)$, we choose constants $\beta$ and
$\gamma$ with $\frac{1}{3} < \beta < \gamma < \alpha < \frac{1}{2}$
and consider controlled rough path space $C^{\beta}_{B^H}$. Let
$\kappa(x_0) \in C^{\beta}_{B^H} $ denote the controlled rough path
defined as $\kappa(x_0) := ( x_0 + \sigma(x_0){B^H}_{0,.} , \sigma(x_0))$
and define the set
\[ K := \biggl\{ (Y,Y') \in C^{\beta}_{B^H} \biggl | Y_0 = x_0, Y_0' = \sigma(x_0), ||(Y,Y') - \kappa(x_0)||_{C^{\gamma}_{B^H}} \leq 1  \biggr\}\]   

Following easy manipulations, the set $K$ can also be written as 
\[ K = \biggl\{ (Y,Y') \in C^{\beta}_{B^H} \biggl | Y_0 = x_0, Y_0' = \sigma(x_0), ||(Y,Y') ||_{{B^H}, \gamma} \leq 1  \biggr\}\]  
and Arzela-Ascoli theorem easily implies that $K$ is a compact convex subset of $C^{\beta}_{{B^H}}$.  Let $\Psi$ be as in (\ref{Psi}). We next establish the continuity of the map $\Psi$ defined above. 

\begin{lemma}
The map $\Psi: C^{\beta}_{B^H}(\mathbb{R}^d)\to C^{\beta}_{B^H}(\mathbb{R}^d)$ is continuous.
\end{lemma}

{\em Proof :} We write $\Psi = \Psi_1 + \Psi_2$ where 
\[ \Psi_1(Y, Y') := \biggl( \int_0^t f(Y_r)dr, 0 \biggr) \]

\[ \Psi_2(Y, Y') := \biggl(x_0  + \int_0^t \sigma(Y_r)d\mathbf{{B}}^H_r , \sigma(Y_t) \biggr)\]

Since $\sigma \in C^3_b$, it follows from Theorem $7.5$ in \cite{FH12}
that the map $S_\sigma: C^{\beta}_{{B^H}}(\mathbb{R}^d) \to
C^{\beta}_{{B^H}}(\mathbb{R}^{d\times d })$ defined by $S_\sigma(Y, Y') :=
(\sigma(Y), \sigma'(Y)Y')$ is a continuous map and Theorem $4.10$ in
\cite{FH12} implies $\Psi_2$ is a continuous map. As for the
continuity of map $\Psi_1$, choose $p, q$ with $p = (1 - 2\beta)^{-1}$
and $p^{-1} + q^{-1} = 1$. If $(Y^n, Y^{n'})$ is a sequence of
controlled rough paths converging to $(Y, Y')$, then by H\"older
inequality,
\[ \biggl|\int_s^t \{f(Y_r^n) - f(Y_r) \}dr\biggr|  \leq \biggl\{\int_0^T |f(Y_r^n) - f(Y_r)|^p dr\biggr\}^{\frac{1}{p}} |t-s|^{2\beta}  \]and using the continuity of $f$ and dominated convergence theorem, we observe that 
\[ \biggl|\biggl|\int_0^. \{f(Y_r^n) - f(Y_r) \}dr \biggr|\biggr|_{2\beta} \to 0 \] implying $\Psi_1$ is continuous. 

\qed

From above Lemma, we easily see that the map $\Psi$ restricted to the subset $K \subset C^{\beta}_{{B^H}}$ is continuous. In the next Lemma, we establish the invariance of $K$ under $\Psi$. 

\begin{lemma} There exists $T$ small enough such that $\Psi(K) \subset K $ for all initial values $x_0 \in \mathbb{R}^d$.
\end{lemma} 

{\em Proof :}
Let $(Y, Y' ) $ be an element of $K$. From the definition of $K$, we need to establish that $||(Z, Z')||_{{B^H}, \gamma} = ||Z'||_{\gamma} + ||R^Z||_{2\gamma}  \leq 1$ for $T$ small enough, where \[(Z,Z') = \biggl(x_0 + \int_0^t f(Y_r)dr + \int_0^t \sigma(Y_r)d\mathbf{B}^H_r , \sigma(Y_t) \biggr).\]
To this end, note that $(\sigma(Y), \sigma'(Y)Y')$ is a controlled rough path and 
\begin{align*} \sigma(Y_t) - \sigma(Y_s) = &  \sigma'(Y_s)Y_s'(B^H_{s,t}) + R^{\sigma(Y)}_{s,t} \\  = &  \sigma'(Y_s)Y_{0,s}'(B^H_{s,t})  + \sigma'(Y_s)\sigma(x_0)(B^H_{s,t}) + R^{\sigma(Y)}_{s,t}.
\end{align*} 
Note that since $\sigma$ and $\sigma'$ are bounded and $||Y'||_{\gamma} \leq 1$, we get 
\[ |\sigma'(Y_s)Y_{0,s}'(B^H_{s,t})  + \sigma'(Y_s)\sigma(x_0)(B^H_{s,t})| \leq C_{\sigma}(T^{\gamma} + 1 )||{B^H}||_{\alpha}|t-s|^{\alpha}.\]
Also, by division property, 
\begin{align*} R^{\sigma(Y)}_{s,t} = & \int_0^1\int_0^1 r \sigma''(ur Y_t + u(1-r)Y_s + (1-u)Y_s )( Y_s'B^H_{s,t} + R^Y_{s,t})(Y_s'B^H_{s,t})drdu \\ + & \int_0^1 \sigma'(rY_t + (1-r)Y_s)R^Y_{s,t} dr.
\end{align*}
Again by observing that $ ||Y'||_{\gamma} + ||R^Y||_{2\gamma} \leq 1 $, $\sigma, \sigma', \sigma''$ are bounded and $|B^H_{s,t}| \leq ||{B^H}||_{\alpha}|t-s|^{\alpha}$, we see that $| R^{\sigma(Y)}_{s,t}| \leq C |t-s|^{2\gamma}$ and thus 
\[ | \sigma(Y_t) - \sigma(Y_s) | \leq C (|t-s|^\alpha + |t-s|^{2\gamma} ), \]
giving $||Z'||_{\gamma} \leq \frac{1}{2} $ for $T$ small enough.  \\
For controlling $||R^Z||_{2\gamma}$, note that \[ \biggl|\int_s^t f(Y_r) dr \biggr| \leq ||f||_{\infty} |t-s| \leq ||f||_{\infty}|t-s|^{2\gamma}T^{1-2\gamma}\]
and by Theorem $4.10$ in \cite{FH12}, \[ \biggl| \int_s^t \sigma(Y_r)d\mathbf{{B}}^H_r - \sigma(Y_s){B}^H_{s,t} - \sigma'(Y_s)Y_s'\mathbb{{B}}^H_{s,t}\biggr| \leq C (||{B^H}||_{\gamma}||R^{\sigma(Y)}||_{2\gamma} + ||\mathbb{{B}}^H||_{2\gamma}||\sigma'(Y)Y'||_{\gamma}) |t-s|^{3\gamma} \]
As shown above, $ ||R^{\sigma(Y)}||_{2\gamma} +  ||\sigma'(Y)Y'||_{\gamma} $  remain bounded over $(Y, Y')\in K$ and we see that 
\begin{align*} \biggl| \int_s^t \sigma(Y_r)d\mathbf{{B}}^H_r - \sigma(Y_s)B^H_{s,t} \biggr| \leq & C ( |t-s|^{3\gamma} + | \sigma'(Y_s)Y_s'\mathbb{{B}}^H_{s,t}| ) \\ \leq & C( |t-s|^{3\gamma} + ||\mathbb{{B}}^H||_{2\alpha}|t-s|^{2\alpha} )
\end{align*}
and thus $||R^Z||_{2\gamma} \leq \frac{1}{2}$ for $T$ small enough, which concludes the proof. 
\qed

As an immediate consequence of the above preparation, we get the existence of solution to \eqref{modelRDE}.

\begin{proof}[Proof of existence of solutions in  Theorem \ref{gen-catellier-gubinelli}] We first view $\Psi$ as a map $\Psi: C^{\beta}_{{B^H}}(\mathbb{R}^d) \to C^{\beta}_{{B^H}}(\mathbb{R}^d)$. From the above Lemmas, $K$ is a compact convex subset of $C^{\beta}_{{B^H}}(\mathbb{R}^d)$ and $\Psi: K \to K $ is a continuous a map over  $[0,\tilde{T}]$ for $\tilde{T}$ small enough. Thus by Schauder's fixed point theorem, $\Psi$ has a fixed point  $ (X, X') \in C^{\beta}_{{B^H}}(\mathbb{R}^d)$. Also since $\tilde{T}$ is not dependent on initial value $x_0$, we get a global solution on $[0,T]$. Finally, since $\mathbf{{B}}^H$ is an $\alpha$-rough path, it can be easily verified using Theorem $4.10$ in \cite{FH12} that $(X,X') \in C^{\alpha}_{{B^H}}(\mathbb{R}^d) $.

\end{proof}
\end{appendices}

Siva Athreya,
{\em 8th Mile Mysore Road, Indian Statistical Institute,
         Bangalore 560059, India.
         Email: \tt{athreya@isibang.ac.in}}

Suprio Bhar,
{\em Tata Institute of Fundamental Research, Centre For Applicable Mathematics, Post Bag No 6503,
    GKVK Post Office, Sharada Nagar, Chikkabommsandra, Bangalore 560065, Karnataka, India.
    Email: \tt{suprio@tifrbng.res.in, suprio.bhar@gmail.com}}

Atul Shekhar,
{\em 8th Mile Mysore Road, Indian Statistical Institute,
         Bangalore 560059, India.
    Email: \tt{atulshekhar83@gmail.com}}


\begin{thebibliography}{99}

 
\bibitem[BKR01]{BKR} V. I. Bogachev, N. V. Krylov , and M. Rockner, { On regularity of transition functions and invariant measures of singular diffusions under minimal conditions,} {\em Comm. PDE 26, 2037-2080,}  (2001). 

 \bibitem[BM10]{BM}O. Butkovsky and L. Mytnik,  Regularization by noise and flows of solutions for a stochastic heat equation {\em https://arxiv.org/abs/1610.02553}

\bibitem[BP10]{BP10} T. M.-Brandis and F. Proske. Construction of strong solutions of SDE’s via Malliavin calculus. {\em Journal of Functional Analysis, 258(11):3922–3953, 2010.}

 \bibitem[CG12]{CG} R. Catellier and M. Gubinelli,  Averaging along irregular curves and regularisation of ODEs.
 {\em https://arxiv.org/abs/1205.1735}
 
\bibitem[CQ02]{Coutin-Qian} Coutin, L. ; Qian Z.: Stochastic analysis, rough path analysis and fractional Brownian motions. Probab. Theory Related Fields 122, no. 1, (2002), 108-140. 
 
\bibitem[D07]{davie} A.M. Davie, Uniqueness of solutions of stochastic differential equations. {\em Int. Math. Res. Not. IMRN (24), Art. ID rnm124, 26.} (2007)


\bibitem[D10]{davie1} A.M. Davie, Individual path uniqueness of solutions of stochastice differential equations. {\em Stochastic analysis, 213--225, DOI 10.1007/978-3-642-15358-7 10, Springer-Verlag Berlin Heidelberg} (2010)

  
\bibitem[F11]{F} F. Flandoli, Random perturbation of PDEs and fluid dynamic models, {\em   Lecture Notes in Mathematics,  Springer, Heidelberg, 2015}  (2011)

\bibitem[FF12]{Flandoli-Fedrizzi} E. Fedrizzi and F. Flandoli, H\"older Flow and Differentiability for SDEs with Nonregular Drift {\em Stochastic Analysis and its Applications 31:4, 708-736, DOI:10.1080/07362994.2012.628908} (2012)

\bibitem[FGP13]{FGP} F. Flandoli, M. Gubinelli, E. Priola, Remarks on the stochastic transport equation with H\"older drift {\em https://128.84.21.199/abs/1301.4012.}

\bibitem[F81]{follmer} H. F\"ollmer. Calcul d'It\^o sans probabilit\'es. In Seminar on Probability, XV (Univ. Strasbourg, Strasbourg,
1979/1980) (French), volume 850 of Lecture Notes in Math., pages 143–150. Springer, Berlin, 1981


\bibitem[FV07]{FV07} P. Friz and  N. Victoir Differential equations driven by {G}aussian signals, {\em Annales de l'Institut Henri Poincar\'{e} (B) Probability and Statistics, Vol. 46, No. 2, 369--413, } (2010)

\bibitem[FH14]{FH12} P. Friz and  Hairer, M., A course on rough paths. {\em Springer Universitext, ISBN 978-3-319-08331-5} (2014)

\bibitem[FS13]{FS1} P. Friz and A. Shekhar, Doob Meyer for Rough Paths, {\em Bulletin of the Institute of Mathematics
Academia Sinica (New Series)
Vol. 8 (2013), No. 1, pp. 73-84}

  
\bibitem[FV10]{FV10} Friz, P.; Victoir, N.: Multidimensional Stochastic Processes
as Rough Paths, Cambridge Studies in Advanced Mathematics Vol 120, Cambridge
University Press, 2010

  



\bibitem[K95]{karan} Rajeeva L. Karandikar. On pathwise stochastic integration. Stochastic Process. Appl., 57(1):11–18, 1995.

\bibitem[KR05]{KR} N. V. Krylov and M. Roeckner. Strong solutions of stochastic equations with singular time dependent drift. {\em Probability theory and related fields, 131(2):154–196, 2005.}
  
\bibitem[L98]{L98} Lyons, T.: Differential equations driven by rough signals,
  Rev. Mat. Iberoamericana 14, no. 2, 215--310, 1998
  
\bibitem[LQ02]{LQ02} T. Lyons and Z. Qian. System control and rough paths. {\em Numerical Methods and Stochastics, 34:91, 2002.}
  

\bibitem[NO02]{N02} D. Nualart and Y. Ouknine. Regularization of differential equations by fractional noise. {\em Stochastic Processes and their Applications, 102(1):103–116, 2002.}


\bibitem[P04]{protter-book} Philip E. Protter, Stochastic Integration and Differential Equations. {\em Springer} (2004)

\bibitem[RY99]{RY} D. Revuz and M. Yor, Continuous Martingales and Brownian Motion, {\em  Grundlehren der mathematischen Wissenschaften, 3rd Edition, Springer-Verlag New York} (1999)

\bibitem[S14]{S1}A.V. Shaposhnikov, Some remarks on Davie's uniqueness theorem {\em https://arxiv.org/abs/1401.5455}

\bibitem[S17]{S2} A.V. Shaposhnikov,  Correction to the paper "Some remarks on Davie's uniqueness theorem"   {\em https://arxiv.org/abs/1703.06598}
 
 \bibitem[V81]{V81} A. Yu. Veretennikov, On strong solutions and explicit formulas for solutions of stochastic integral equations {\em (Russian), Mat. Sbornik (N.S.) 111(153) (1980), 434-452, 480. English transl. in Math. USSR Sb. 39 (1981), 387-403.}

  

\end{thebibliography}
\end{document}